\documentclass[reqno]{amsart}
\usepackage{amssymb, latexsym,amsmath,amsthm,enumerate}
\usepackage{orcidlink}
\usepackage[mathscr]{eucal}
\usepackage{framed,color,graphicx}
\usepackage{mathrsfs}
\usepackage[all]{xy} 
\usepackage{tikz}
\usetikzlibrary{positioning}

\usepackage{appendix}

\DeclareSymbolFont{bbold}{U}{bbold}{m}{n}
\DeclareSymbolFontAlphabet{\mathbbold}{bbold}

\theoremstyle{plain} 
\newtheorem{thm}{Theorem}
\newtheorem{lem}[thm]{Lemma}

\usepackage{thmtools}
\usepackage{thm-restate}

\theoremstyle{definition}

\newtheorem{remark}[thm]{Remark}

\newcommand{\varmem}{\textsf{varmem}}

\newcommand{\up}[1]{\textup{#1}}

\newcommand{\join}{\mathbin{\ooalign{\hfil$\sqcup$\hfil\cr\hfil\raise0.28ex\hbox{\tiny\textup{+}}\hfil\cr}}}

\makeatletter
\newcommand*\bigcdot{\mathpalette\bigcdot@{.75}}
\newcommand*\bigcdot@[2]{\mathbin{\vcenter{\hbox{\scalebox{#2}{$\m@th#1\bullet$}}}}}
\makeatother

\begin{document}

\title[]{Even harder pseudovariety membership problem}

\author{Marcel Jackson \orcidlink{0000-0002-8149-1141}}
\address{Department of Mathematical and Physical Sciences, La Trobe University, Victoria  3086,
Australia} \email{m.g.Jackson@latrobe.edu.au}


\begin{abstract}
We give a finite semigroup for which the membership problem for the generated pseudovariety is hard for Difference P.
\end{abstract}

\maketitle

For a finite monoid (or semigroup) ${\bf M}$, let $\varmem({\bf M})$ denote the problem of deciding, when given a finite monoid ${\bf S}$ (or semigroup), whether ${\bf S}$ lies in the pseudovariety $\mathbb{V}({\bf M})$ generated by ${\bf M}$.  
In this note, we derive a new level of hardness for the problem $\varmem({\bf M})$ for monoids, with all arguments holding in either the monoid or semigroup signature.  Throughout, hardness reductions will be logspace many-one reductions.

Let ${\bf J}_k$ denote the monoid of all order-preserving transformations $f$ of the $k+1$-element chain $0\leq 1\leq\dots\leq k$ satisfying $f(i)\geq i$ for all $i$.  
Let ${\bf B}_2^1$ denote the six-element Brandt monoid.  
The first finite monoids with intractable $\varmem$ problem were given in Jackson and McKenzie~\cite{jacmck} (\texttt{NP}-hard).  
More recently Jackson~\cite{jac:B21} showed that $\varmem({\bf B}_2^1)$ is \texttt{NP}-hard
while Kl\'{\i}ma, Kunc and Pol\'ak showed that $\varmem({\bf J}_k)$ is co-\texttt{NP}-complete, for each $k\geq 4$.  
\begin{thm}\label{thm:DP}
For $k\geq 4$ the problem $\varmem({\bf B}_2^1\times {\bf J}_k)$ is \texttt{DP}-hard.
\end{thm}
The complexity class \texttt{DP} (``difference \texttt{P}'') is the second level of the Boolean hierarchy (BH), consisting of problems that are the intersection of a problem in \texttt{NP} and a problem in co-\texttt{NP}.  It contains both \texttt{NP} and co-\texttt{NP}.

Let $I$ be an instance of 3SAT. \ In Jackson~\cite{jac:B21}, a monoid $B_I$ is constructed that lies in $\mathbb{V}({\bf B}_2^1)$ if and only if $I$ is satisfiable; the monoids $B_I$ are extensions of ${\bf B}_2^1$.
(Formally, the reduction in~\cite{jac:B21} is from positive NAE3SAT, but 3SAT reduces to positive NAE3SAT so it is simpler in the present case to start at 3SAT.)
Kl\'{\i}ma, Kunc, Pol\'ak~\cite{KKP} construct a nilpotent monoid $C_I$ that lies in $\mathbb{V}({\bf J}_k)$ if and only if $I$ is \emph{not} satisfiable.   
In both of these results, slightly more can be deduced.
\begin{lem}\label{lem:jac}\up(Jackson~\cite[Theorem~7.3]{jac:B21}.\up)
If $\mathcal{V}$ is a pseudovariety not containing~${\bf B}_2^1$, then $B_I\in\mathcal{V}\vee\mathbb{V}({\bf B}_2^1)$ if and only if $I$ is a YES instance of 3SAT.
\end{lem}
The following is a strengthening of Proposition~6.1 of Kl\'{\i}ma, Kunc, Pol\'ak~\cite{KKP}.
\begin{lem}\label{lem:KKP}
For $k\geq 4$, if $\mathcal{V}$ is a periodic pseudovariety with index $i\leq k-1$ and period $p$, then $C_I\in\mathcal{V}\vee\mathbb{V}({\bf J}_k)$ if and only if $I$ is a NO instance of 3SAT.
\end{lem}
\begin{proof}
Recall that a \emph{scattered subword} of a word ${\bf v}$ is any subsequence of ${\bf v}$.  
We write ${\bf v}\sim_k{\bf w}$ if the words ${\bf v}$ and ${\bf w}$ have the same length-$k$ scattered subwords.
It is known that ${\bf J}_k\models {\bf v}\approx {\bf w}$ if and only if ${\bf v}\sim_k{\bf w}$  (combining Volkov's~\cite[Theorem~2]{vol} with well known results of Simon~\cite{sim}; see \cite[Proposition~3]{vol} for example).

Let $I$ be an instance of 3SAT.
The monoid $C_I$ is defined as the transition monoid of a constructed finite state automata $\mathcal{A}_I$ for a language built from $I$. 
When~$I$ is a NO instance, then $C_I\in \mathbb{V}({\bf J}_k)\subseteq\mathcal{V}\vee\mathbb{V}({\bf J}_k)$ by \cite[Proposition~6.1]{KKP}.
When~$I$ is a YES instance, the proof of \cite[Proposition~6.1]{KKP} shows that there is a word ${\bf v}_I$ with prefix $x^{k-2}$ such that ${\bf v}_I\sim_k x{\bf v}_I$ but where ${\bf v}_I$ is accepted by~$\mathcal{A}_I$ and $x{\bf v}_I$ is not ($x$ is the letter $c$ in \cite{KKP}).  It follows that ${\bf J}_k\models {\bf v}_I\approx x{\bf v}_I$ but $C_I\not\models {\bf v}_I\approx x{\bf v}_I$, which is why in this situation we have $C_I\notin \mathbb{V}({\bf J}_k)$.  
It is easily verified that the same is true if we replace $x{\bf v}_I$ with $x^p{\bf v}_I$: non-acceptance by $\mathcal{A}_I$ follows trivially as the only $x^{k-1}$ paths lead to an absorbing non-accepting state (and $x^p{\bf v}_I$ has a prefix $x^{k-2+p}$). 
The property ${\bf v}_I\sim_k x^p{\bf v}_I$ follows for identical reason to the case of $p=1$: there are no new $k$-letter scattered subwords obtained by adding further copies of $x$ to the front.  
So both $\mathcal{V}$ (by virtue of $x^{k-2}\approx x^{k-2+p}$) and $\mathbb{V}({\bf J}_k)$ (by virtue of ${\bf v}_I\sim_k x^p{\bf v}_I$) satisfy ${\bf v}_I\approx x^p{\bf v}_I$, which fails on $C_I$.
\end{proof}
\begin{proof}[Proof of Theorem~\ref{thm:DP}]
We have a reduction from the \texttt{DP}-complete problem SAT-UNSAT, where instances consist of a pair of instances $(I,J)$ of SAT (we use 3SAT, which is equivalent), and we wish to know if $I$ is a YES instance and $J$ a NO instance of 3SAT.  We argue that $B_I\times C_J\in \mathbb{V}({\bf B}_2^1\times {\bf J}_k)$ if and only if $(I,J)$ is a YES instance of SAT-UNSAT.

For the backward implication, if $I$ is a YES instance and $J$ a NO instance, then $B_I\in \mathbb{V}({\bf B}_2^1)\subseteq \mathbb{V}({\bf B}_2^1\times {\bf J}_k)$ by Lemma \ref{lem:jac} and $C_J\in \mathbb{V}({\bf J}_k)\subseteq \mathbb{V}({\bf B}_2^1\times {\bf J}_k)$ by Lemma~\ref{lem:KKP}.  

For the forward implication, assume either $I$ is a NO instance or $I$ is a YES instance of 3SAT.  
We use that fact that both $B_I$ and $C_J$ are contained in the variety of $B_I\times C_J$.  
If $I$ is a NO instance, then $B_I\notin  \mathbb{V}({\bf B}_2^1\times {\bf J}_k)$ by Lemma~\ref{lem:jac} because ${\bf J}_k$ is $\mathscr{J}$-trivial; so $B_I\times C_J\notin \mathbb{V}({\bf B}_2^1\times {\bf J}_k)$.
If $J$ is a YES instance, then $C_J\notin \mathbb{V}({\bf B}_2^1\times {\bf J}_k)$ by Lemma~\ref{lem:KKP} as ${\bf B}_2^1$ has index $2$; so $B_I\times C_J\notin \mathbb{V}({\bf B}_2^1\times {\bf J}_k)$.
\end{proof}
\begin{remark}
Volkov~\cite{vol} has shown that ${\bf J}_k$ generates the pseudovariety corresponding to $k$-piecewise testable languages (via syntactic monoids), and that this can also be generated by the monoid of reflexive binary relations on a $(k+1)$-element set (equivalently, $(k+1)\times (k+1)$ Boolean matrices with $1$ on the diagonal), and also by submonoid of upper triangular Boolean matrices with $1$ on the diagonal.  So each of these can replace ${\bf J}_k$ in Theorem~\ref{thm:DP}.
\end{remark}
The proof of Theorem~\ref{thm:DP} does not preclude the possibility that $\varmem({\bf B}_2^1\times {\bf J}_k)$ lies outside $\texttt{DP}$, in some even ``harder'' class.
Amongst general algebras, it is known that $\varmem({\bf A})$ can be as hard as 2-\texttt{exptime}-complete \cite{koz}, and it is tempting to speculate that hardness beyond $\texttt{DP}$ is possible for monoids. 


\end{document}